\documentclass[a4paper,11pt,reqno]{amsart}
\usepackage{graphics,enumitem,epsfig,textcomp}
\usepackage{amsfonts,amsmath,amssymb,euscript,color,mathrsfs}
\usepackage[utf8]{inputenc}	

\usepackage[caption=false,subrefformat=parens,labelformat=parens]{subfig}

\usepackage{url}

\usepackage[symbol]{footmisc}

\newtheorem{lemma}{Lemma}%[section]
\newtheorem{remark}{Remark}%[section]

\def\sign{{\rm sign}}

\def\d{\,\mathrm{d}}
\def\eps{\varepsilon}
\def\N{\mathbb{N}}
\def\R{\mathbb{R}}
\def\C{\hbox{\rlap{\kern.24em\raise.1ex\hbox
      {\vrule height1.3ex width.9pt}}C}}
\def\P{\hbox{\rlap{I}\kern.16em P}}
\def\Q{\hbox{\rlap{\kern.24em\raise.1ex\hbox
      {\vrule height1.3ex width.9pt}}Q}}
\def\M{\hbox{\rlap{I}\kern.16em\rlap{I}M}}
\def\Z{\hbox{\rlap{Z}\kern.20em Z}}
\def\({\begin{eqnarray}}
\def\){\end{eqnarray}}
\def\[{\begin{eqnarray*}}
\def\]{\end{eqnarray*}}
\def\part#1#2{\frac{\partial #1}{\partial #2}}

\def\grad{\nabla}
\def\pmb#1{\setbox0=\hbox{$#1$}
  \kern-.025em\copy0\kern-\wd0
  \kern-.05em\copy0\kern-\wd0
  \kern-.025em\raise.0433em\box0 }

\def\tot#1#2{\frac{\d #1}{\d #2}} 
\def\laplace{\Delta}
\def\d{\,\mathrm{d}}
\def\N{\mathbb{N}}
\def\R{\mathbb{R}}

\def\epsilon{\varepsilon}

\def\E{\mathcal{E}}
\def\P{\mathbb{P}}
\def\Q{\mathbb{Q}}

\newcommand*\di{\mathop{}\!\mathrm{d}}
\newcommand\bl{\left(}
\newcommand\br{\right)}

\newcommand\Vset{V}
\newcommand\Eset{E}

\newcommand\Neigh{\mathcal{N}}

\begin{document}

\centerline{{\Large\textbf{Murray's law for discrete and continuum models of biological networks}}}
\vskip 10mm

%%%%%%%%%%%%%%%%%%%%%%%%%%%%%%%%%
%% Authors
%%%%%%%%%%%%%%%%%%%%%%%%%%%%%%%%%
\centerline{
	{\large Jan Haskovec}\footnote{Mathematical and Computer Sciences and Engineering Division,
		King Abdullah University of Science and Technology,
		Thuwal 23955-6900, Kingdom of Saudi Arabia; 
		{\it jan.haskovec@kaust.edu.sa}}\qquad
	{\large Peter Markowich}\footnote{Mathematical and Computer Sciences and Engineering Division,
		King Abdullah University of Science and Technology,
		Thuwal 23955-6900, Kingdom of Saudi Arabia;
		{\it peter.markowich@kaust.edu.sa}, and
		Faculty of Mathematics, University of Vienna, Oskar-Morgenstern-Platz 1, 1090 Vienna;
		{\it peter.markowich@univie.ac.at}}\qquad
	{\large Giulia Pilli}\footnote{Faculty of Mathematics, University of Vienna, Oskar-Morgenstern-Platz 1, 1090 Vienna;
		{\it giulia.pilli@univie.ac.at}}
	}
\vskip 10mm

\noindent{\bf Abstract:}
We demonstrate the validity of Murray's law,
which represents a scaling relation for branch conductivities in a transportation network,
for discrete and continuum models of biological networks.
We first consider discrete networks with general metabolic coefficient and multiple branching nodes
and derive a generalization of the classical 3/4-law.
Next we prove an analogue of the discrete Murray's law for the continuum system
obtained in the continuum limit of the discrete model on a rectangular mesh.
Finally, we consider a continuum model derived from phenomenological considerations and show
the validity of the Murray's law for its linearly stable steady states.
\vskip 7mm

\vskip 5mm

\vskip 12mm

\section{Introduction}
Murray's law %\cite{Murray1, Murray2}
is a basic physical principle for transportation networks which
predicts the thickness or conductivity of branches,
such that the cost for transport and maintenance of the transport medium is minimized.
This law is observed in the vascular and respiratory systems of animals, xylem in plants,
and the respiratory system of insects\cite{Sherman, Razavi}.
It is also a powerful biomimetics design tool in engineering and has been applied in the design
of microfluidic devices, self-healing materials, batteries, photocatalysts, and gas sensors;
see, e.g., Refs.~\cite{Alston, Emerson, Stephenson, Zheng}.

Murray's original analysis \cite{Murray1, Murray2} is based on the assumption that the radii inside a lumen-based system
are such that the work for transport and upkeep is minimized. Larger vessels lower the expended energy for transport (pumping power),
but increase the overall volume of blood in the system, which requires metabolic support.
Murray's law reflects the optimum balancing of the pumping power versus metabolic expenditure.
Assuming laminar flow in a single tube, the pumping power $E_p$ is given by the Joule's law,
\[
   E_p = Q\Delta P,
\]
where $Q$ is the volumetric flux and $\Delta P$ is the pressure difference between the entry and exit of the tube.
According to the Hagen-Poiseuille law for laminar flow in a three-dimensional tube,
\(  \label{Q}
   Q = \frac{\pi R^4}{8\mu L} \Delta P,
\)
where $R>0$ is the radius of the tube, $L>0$ its length, and $\mu>0$ the dynamic viscosity of the fluid.
Therefore, we have
\[
   E_p = \frac{8\mu L}{\pi R^4} Q^2.
\]
Furthermore, we assume that the metabolic cost $E_m$ is proportional to the number of blood cells
in the tube\cite{Murray1}, and that is proportional to the tube's volume,
\[
   E_m = \nu \pi L R^2,
\]
where the proportionality constant $\nu>0$ is called the metabolic coefficient.
This gives for the total energy
\(  \label{E}
   E := E_p + E_m = \frac{8\mu L}{\pi R^4} Q^2 + \nu \pi L R^2.
\)
A trivial calculation shows that the total energy is minimized
if the following relation between the flow rate and the tube radius is verified,
\(   \label{Q-r}
   Q = \sqrt{\frac{\pi^2 \nu}{16\mu}} R^3.
\)
In a branching node with a parent branch with flow rate $Q_0>0$
and children branches with flow rates $Q_1, Q_2,\dots,Q_n>0$,
we have the flux balance
\[
   Q_0 = \sum_{i=1}^n Q_i.
\]
%where $S\in\R$ denotes the external in- or outflow through the node.
Using \eqref{Q-r}, we obtain for the respective radii $R_0>0$ and $R_1, R_2, \dots, R_n>0$,
\(  \label{M1}
   R_0^3 = \sum_{i=1}^n R_i^3,
\)
which is the classical Murray's law for laminar blood flow\cite{Murray1, Murray2}.
It can be also expressed in terms of the edge conductivities $C_i>0$, which for three-dimensional laminar flow
are proportional to $R_i^4$, see Ref.~\cite{Rogers}, therefore
\(  \label{M2}
   C_0^{3/4} = \sum_{i=1}^n C_i^{3/4}.
\)
The above formula represents the classical 3/4-law
for optimal ratio of conductivities in 3D branching edges
with laminar flow \cite{Murray1,Murray2}.
It clearly demonstrates the nonlinear
nature of optimal network conductivity models,
in particular, the nonlinear (and nonconvex) dependence of the
total energy \eqref{E} on the edge conductivity.

The first goal of this paper is to show the validity of a generalization of the Murray's law \eqref{M2}
%reformulated in terms of edge conductivities,
for general transportation networks with multiple branching nodes (Section \ref{sec:discrete}).
%(note that \eqref{M1} was only derived for a single branching node).
Then, in Section \ref{sec:cont1}, we consider a continuum limit of the discrete model
on a rectangular mesh, and derive an analogue of the discrete Murray's law for the
continuum PDE system.
Finally, in Section \ref{sec:cont2} we consider another continuum model
for biological transportation networks, derived from phenomenological considerations,
and show that a continuum version of Murray's law holds also for the linearly stable solutions of this model.
The calculations made in Sections \ref{sec:cont1} and \ref{sec:cont2} are formal,
however, they can be justified by assuming sufficient regularity of the steady state solutions
of the respective continuum PDE systems. Details will be provided in the respective Sections.

%%%%%%%%%%%%%%%%%%%%%%%%%%%%%%%%%%%%%%%%
\section{Murray's law for the discrete network model}\label{sec:discrete}
We consider a given undirected connected graph $G=(\Vset,\Eset)$,
consisting of a finite set of vertices (branching nodes) $\Vset$ of size $N=|\Vset|$ and a finite set of edges $\Eset$.
Any pair of vertices is connected by at most one edge, which represents a tube (channel),
and no vertex is connected to itself.
We denote the edge  between vertices $i\in\Vset$ and $j\in\Vset$ by $(i,j)\in \Eset$.
Since the graph is undirected, $(i,j)$ and $(j,i)$ refer to the same edge.
For each edge $(i,j)\in\Eset$ of the graph $G$ we consider its length and its conductivity, denoted by  $L_{ij}=L_{ji}>0$ and $C_{ij}=C_{ji}\geq 0$, respectively. 
%We denote the conductivity vector by $C=(C_{ij})_{(i,j)\in\Eset}$. 
The edge lengths $L_{ij}>0$ are given as a datum and fixed for all $(i,j)\in\Eset$.
With each vertex $i\in\Vset$ there is associated the fluid pressure $P_i\in\R$.
The pressure drop between vertices $i\in\Vset$ and $j\in\Vset$ connected by an edge $(i,j)\in\Eset$ is denoted by 
\begin{align}\label{eq:pressuredrop}
(\Delta P)_{ij}:=P_j-P_i.
\end{align}
Note that the pressure drop is antisymmetric, i.e., by definition,  $(\Delta P)_{ij}=-(\Delta P)_{ji}$.
The oriented flux (flow rate) from vertex $j\in\Vset$ to $i\in\Vset$ is denoted by $Q_{ij}$; again, we have $Q_{ij}=-Q_{ji}$.
For the three-dimensional laminar flow, conductance (transportation capacity) $C_{ij}$ of an edge $(i,j)\in\Eset$ is proportional
to the fourth power of its diameter\cite{Rogers}. Then, \eqref{Q} is reformulated in terms of the conductivities as
\begin{align}\label{eq:flowrate}
   Q_{ij} := C_{ij}\frac{P_j-P_i}{L_{ij}}\qquad\text{for all~}(i,j)\in \Eset.
\end{align}
The flux balance in each vertex is expressed in terms of the Kirchhoff law
\begin{align}\label{eq:kirchhoff}
   -\sum_{j\in \Neigh(i)} C_{ij}\frac{ P_j-P_i}{L_{ij}}=S_i\qquad \text{for all~}i\in \Vset.
\end{align}
Here $\Neigh(i)\subset\Vset$ denotes the set of vertices connected to $i\in\Vset$ through an edge, i.e.,
\[
   \Neigh(i) := \{ j\in\Vset;\, (i,j)\in\Eset\},
\]
and $S=(S_i)_{i\in\Vset}$ is the prescribed strength of the external flow source ($S_i>0$) or sink ($S_i<0$) in node $i\in\Vset$.
Clearly, a necessary condition for the solvability of \eqref{eq:kirchhoff} is the global mass conservation
\begin{align}\label{eq:conservationmass}
   \sum_{i\in\Vset}S_i=0,
\end{align}
which we shall assume in the sequel.
Given the vector of conductivities $C=(C_{ij})_{(i,j)\in\Eset}$, the Kirchhoff law \eqref{eq:kirchhoff} is a linear system of equations for the vector of pressures $P=(P_i)_{i\in\Vset}$.
With the global mass conservation \eqref{eq:conservationmass}, the linear system \eqref{eq:kirchhoff}
is solvable if and only if the graph with edge weights $C=(C_{ij})_{(i,j)\in\Eset}$ is connected\cite{bookchapter},
where only edges with positive conductivities $C_{ij}>0$ are taken into account
(i.e., edges with zero conductivities are discarded).
Note that the solution vector $P=(P_i)_{i\in\Vset}$ is unique up to an additive constant.
The corresponding fluxes, which we denote by $Q_{ij}[C]$, are then calculated uniquely
from \eqref{eq:flowrate}.

The conductivities $C_{ij}$ are subject to an energy optimization and adaptation process,
where, in analogy to \eqref{E}, the cost functional is a sum of pumping and metabolic energies
over all edges of the graph,
\begin{align}\label{eq:energydisc}
   {\E}[C] := \sum_{(i,j)\in\Eset}\bl \frac{Q_{ij}[C]^2}{C_{ij}}+\frac{\nu}{\gamma} C_{ij}^{\gamma}\br L_{ij}.
\end{align}
Here we consider a generalization where the metabolic energy is proportional to a power $\gamma>0$ of the respective edge conductivity.
We have $\gamma=1/2$ for blood vessel systems (recall that in a three-dimensional tube the conductivity is proportional to fourth power of the tube radius,
while the metabolic cost scales quadratically with the radius, see Ref.~\cite{Murray1}).
For models of plant leaf venation the metabolic cost is proportional to the number of small tubes, which is
proportional to $C_{ij}$, and the metabolic cost is due to the effective loss of the photosynthetic power
at the area of the venation cells, which is proportional to $C_{ij}^{1/2}$. Consequently, the effective value
of $\gamma$ typically used in models of leaf venation in plants lies between $1/2$ and $1$, see, e.g., Refs.~\cite{Hu, Hu-Cai}.
Again, $\nu>0$ is the so-called metabolic coefficient.
%Note that  every edge of the graph $G$ is counted exactly once in the above sum.

The derivation of the Murray's law for critical points of the energy functional \eqref{eq:energydisc},
constrained by the Kirchhoff law \eqref{eq:kirchhoff}, is based on
the following explicit formula for the derivative of the pumping term in \eqref{eq:energydisc} %constrained by  Kirchhoff's law \eqref{eq:kirchhoff},
with respect to the conductivities, derived in Ref.~\cite{HKM}. We restate its proof here for the sake of completeness.

\begin{lemma}\label{lem:derivative}
        Let $C=(C_{ij})_{(i,j)\in\Eset}$ be a vector of nonnegative conductivities such that the underlying graph is connected
        if only edges with strictly positive conductivities are taken into account.
	Let $Q_{ij}[C]$ be given by \eqref{eq:flowrate}, where $P=(P_i)_{i\in\Vset}$ is a solution
	of the linear system \eqref{eq:kirchhoff} with the vector of conductivities $C$.
	Then, for any fixed $(k,l)\in E$ we have
	\(  \label{eq:derivativeq}
	\part{}{C_{kl}} \sum_{(i,j)\in E} \frac{Q_{ij}[C]^2}{C_{ij}}L_{ij} = - \frac{Q_{kl}[C]^2}{C_{kl}^2}L_{kl}.
	\)
\end{lemma}

\begin{proof}
Since
\begin{align*}
\part{}{C_{kl}} \sum_{(i,j)\in E} \frac{Q_{ij}[C]^2}{C_{ij}}L_{ij}=- \frac{Q_{kl}[C]^2}{C_{kl}^2}L_{kl}+2\sum_{(i,j)\in E} \frac{Q_{ij}[C]}{C_{ij}}\frac{\partial Q_{ij}[C]}{\partial C_{kl}}L_{ij},
\end{align*}
it is sufficient to show that
\begin{align*}
\sum_{(i,j)\in E} \frac{Q_{ij}[C]}{C_{ij}}\frac{\partial Q_{ij}[C]}{\partial C_{kl}}L_{ij}=0.
\end{align*}
%by the definition of the flow rate $Q_{ij}$ in \eqref{eq:flowrate}.
Let $\mathbb{A}=(\mathbb{A}_{ij})$ denote the adjacency matrix of the graph $G=(\Vset,\Eset)$, i.e.\ its  coefficients are defined by
\begin{align}\label{adjM}
\mathbb{A}_{ij}=\begin{cases}
0 & \mbox{if }(i,j)\notin \Eset, \\
1 & \mbox{if }(i,j)\in \Eset.
\end{cases}
\end{align}
Note that $G$ is an undirected graph, implying $\mathbb{A}_{ij}=\mathbb{A}_{ji}$.
%, $C_{ij}=C_{ji}$, $L_{ij}=L_{ji}$ and $Q_{ij}=-Q_{ji}$. For $(i,j)\in\Eset$ we have
%Using the assumption that  vertices are not connected with themselves, i.e.\  $\mathbb{A}_{ii}=0$, we have
Due to the symmetry of $C_{ij}$ and $L_{ij}$ and antisymmetry of $Q_{ij}$ we have
\begin{align*}
2\sum_{(i,j)\in E} \frac{Q_{ij}[C]}{C_{ij}}\frac{\partial Q_{ij}[C]}{\partial C_{kl}}L_{ij}
&=\sum_{i=1}^n\sum_{j=1}^n \mathbb{A}_{ij}\bl \frac{P_j-P_i}{L_{ij}}\frac{\partial Q_{ij}[C]}{\partial C_{kl}}\br L_{ij}
\\&=\sum_{j=1}^n P_j\sum_{i=1}^n \mathbb{A}_{ij} \frac{\partial Q_{ij}[C]}{\partial C_{kl}} -\sum_{i=1}^n P_i\sum_{j=1}^n \mathbb{A}_{ij}  \frac{\partial Q_{ij}[C]}{\partial C_{kl}} 
\\&=-2\sum_{i=1}^n P_i\sum_{j=1}^n \mathbb{A}_{ij} \frac{\partial Q_{ij}[C]}{\partial C_{kl}} 
\\&=-2\sum_{i=1}^n P_i\frac{\partial}{\partial C_{kl}} \sum_{j\in N(i)} Q_{ij}.
\end{align*}
By the definition of the flow rate $Q_{ij}$ in \eqref{eq:flowrate} and Kirchhoff's law \eqref{eq:kirchhoff} we have
\begin{align*}
 - \sum_{j\in N(i)}Q_{ij }=\sum_{j\in N(i)} C_{ij} \frac{P_j-P_i}{L_{ij}}=S_i,
\end{align*}
and since the sources/sinks $S_i$ are fixed, we conclude
\begin{align*}
\sum_{(i,j)\in E} \frac{Q_{ij}[C]}{C_{ij}}\frac{\partial Q_{ij}[C]}{\partial C_{kl}}L_{ij}=0.
\end{align*}
\end{proof}

Using the identity \eqref{eq:derivativeq}, it is straightforward to calculate the derivative of the energy \eqref{eq:energydisc},
constrained by the Kirchhoff law \eqref{eq:kirchhoff}, with respect to the edge conductivities,
\begin{align*}
   \frac{\partial}{\partial C_{ij}} {\E}[C] = - \left( \frac{Q_{ij}[C]^2}{C_{ij}^2} - \nu C_{ij}^{\gamma-1} \right) L_{ij}.
\end{align*}
Since any critical point of ${\E}[C]$ satisfies
\[
   \frac{\partial}{\partial C_{ij}} {\E}[C] = 0\qquad\mbox{for all } (i,j)\in\Eset,
\]
we have
\(   \label{eq:crit}
   Q_{ij}[C]^2 = \nu C_{ij}^{\gamma+1}\qquad\mbox{for all } (i,j)\in\Eset.
\)
Now, the Kirchhoff law \eqref{eq:kirchhoff} written in terms of the fluxes $Q_{ij}:=Q_{ij}[C]$ reads
\(   \label{kiQ}
    \sum_{j\in \Neigh(i)} Q_{ij} + S_i = 0\qquad \text{for all~}i\in \Vset.
\)
Let us write the set of neighbors $\Neigh(i)$ of the node $i\in\Vset$ as the disjoint union $\Neigh(i) = \Neigh^+(i) \cup \Neigh^-(i)$,
based on the flow directions along the respective edges, i.e.,
\[
   \Neigh^+(i) := \{ j\in \Neigh(i); \; Q_{ij} \geq 0 \},\qquad
   \Neigh^-(i) := \{ j\in \Neigh(i); \; Q_{ij} < 0 \}.
\]
Then, we can rewrite \eqref{kiQ} as
\[
    \sum_{j\in \Neigh^+(i)} Q_{ij} + S_i = \sum_{j\in \Neigh^-(i)} (-Q_{ij}) \qquad \text{for all~}i\in \Vset,
\]
and, further,
\[
    \sum_{j\in \Neigh^+(i)} |Q_{ij}| + S_i = \sum_{j\in \Neigh^-(i)} |Q_{ij}| \qquad \text{for all~}i\in \Vset.
\]
Using \eqref{eq:crit}, we arrive at
\(  \label{Mur1}
    \sqrt{\nu} \sum_{j\in \Neigh^+(i)}  C_{ij}^\frac{\gamma+1}{2} + S_i = \sqrt{\nu} \sum_{j\in \Neigh^-(i)} C_{ij}^\frac{\gamma+1}{2} \qquad \text{for all~}i\in \Vset.
\)
This is the (generalized) Murray's law for the discrete model \eqref{eq:kirchhoff}, \eqref{eq:energydisc}.
Recall that for the case of three-dimensional blood vessel systems we have $\gamma=1/2$,
which leads to the power $\frac{\gamma+1}{2} = \frac34$,
so that \eqref{Mur1} becomes a direct generalization of the classical Murray's law \eqref{M2}.

%%%%%%%%%%%%%%%%%%%%%%%%%%%%%%%%%%%%%%%%%%%%%%%%%%%%%%
\section{Murray's law for the continuum limit model on rectangular grids}\label{sec:cont1}
In Ref.~\cite{HKM} the formal continuum limit of of the discrete model \eqref{eq:kirchhoff}, \eqref{eq:energydisc} for $\gamma>1$
on rectangular equidistant grids was derived; the rigorous limit passage was studied in the consequent paper\cite{HKM2}.
The resulting continuum energy functional is of the form
\begin{align}  \label{Econt}
   \E[c] := \int_{\Omega} \grad p\cdot c\grad p + \frac{\nu}{\gamma} |c|^\gamma  \d x,
\end{align}
with the {metabolic coefficient} $\nu>0$ and the {exponent} $\gamma>1$.
The energy functional $\E[c]$ is defined on the set of nonnegative diagonal tensor fields $c=c(x)$ on a bounded domain $\Omega\subset\R^d$,
$d\in\N$,
\( \label{cTensor}
   c = \begin{pmatrix} c^1 & & \\ & \ddots & \\ & & c^d \end{pmatrix}.
\)
The symbol $|c|^\gamma$ is defined as $\sum_{k=1}^d \left| c^k \right|^\gamma$.
The scalar pressure $p=p(x)$ of the fluid occupying the domain $\Omega$ %within the network (porous medium)
is subject to the Poisson equation
\begin{align} \label{Pcont}
   -\grad\cdot(c\grad p) = S,
\end{align}
equipped with no-flux boundary condition on $\partial\Omega$, and the datum $S=S(x)$ represents the intensity of sources and sinks in $\Omega$.
Let us point out that the Poisson equation \eqref{Pcont} is possibly strongly degenerate
since the eigenvalues (i.e., diagonal elements) of the permeability tensor \eqref{cTensor} may vanish.
To overcome this problem, we follow Ref.~\cite{HKM} and introduce a regularization of \eqref{Pcont} of the form
\begin{align} \label{PcontR2}
   -\grad\cdot(\P[c]\grad p) = S,
\end{align}
with the permeability tensor
\( \label{P}
   \P[c] := rI + c,
\)
where $r=r(x) \geq r_0 > 0$ is a prescribed function that models the isotropic background permeability of the medium,
and $I\in\R^{d\times d}$ is the unit matrix.
Again, \eqref{PcontR2} is equipped with the no-flux boundary condition
\(  \label{noflux1}
   n\cdot \P[c]\grad p = 0\qquad\mbox{on }\partial\Omega,
\)
where $n=n(x)$ denotes the outer unit normal vector to $\partial\Omega$.
Then, the weak formulation of the boundary value problem \eqref{PcontR2}--\eqref{noflux1}
with a test function $\phi\in C^\infty(\Omega)$ reads
\( \label{weakP}
   \int_\Omega \grad\Phi\cdot (rI + c)\grad p \d x = \int_\Omega S\phi \d x.
\)
Clearly, since we only admit nonnegative diagonal tensor fields $c=c(x)$,
\eqref{weakP} is uniformly elliptic and %the corresponding homogeneous Neumann boundary problem
it has solutions unique up to an additive constant.
The energy functional \eqref{Econt} needs to be updated as
\(   \label{Econt2}
   \E[c] := \int_{\Omega} \grad p\cdot \P[c]\grad p + \frac{\nu}{\gamma} |c|^\gamma  \d x. 
\)
We now calculate the Euler-Lagrange equations corresponding to critical
points of the functional \eqref{Econt2} constrained by \eqref{weakP}. %\eqref{PcontR2}--\eqref{noflux1}.

\begin{lemma}\label{lem:gfcont}
%Let $\gamma>0$.
%The critical points $(c,p)$ of the constrained functional \eqref{Econt2}, \eqref{PcontR2} satisfy
The Euler-Lagrange equations for the constrained energy minimization problem %\eqref{Econt2}, \eqref{PcontR2}--\eqref{noflux1} with $\gamma\neq 1$
\eqref{weakP}, \eqref{Econt2} read
\(   \label{crit1}
   \left( \partial_{x_k} p\right)^2 - \nu  |c^k|^{\gamma-1} = 0,\qquad\mbox{for } k=1,\dots,d.
\)
%With $\gamma=1$ the Euler-Lagrange equations read
%\(   \label{crit11}
%   \left( \partial_{x_k} p\right)^2 - \nu = 0,\qquad\mbox{for } k=1,\dots,d.
%\)
\end{lemma}

\begin{proof}
%Let us first consider the case $\gamma\neq 1$.
We calculate the first variation of $\E$ in the direction $\phi$ where $\phi$ denotes a diagonal matrix with entries $\phi^1,\ldots,\phi^d$,
such that $c+\eps\phi$ is nonnegative.
Using the expansion 
\begin{align}\label{eq:expansioncont}
   p[c+\epsilon\phi]=p_0+\epsilon p_1 + \mathcal{O}(\eps^2),
\end{align}
we have
\begin{align}\label{eq:firstvariation}
   \left.\tot{}{\epsilon}\E[c+\epsilon\phi]\right|_{\epsilon=0} =
      \sum_{k=1}^d\int_{\Omega} \left( \partial_{x_k} p_0\right)^2 \phi^k + 2 (r+c^k) (\partial_{x_k} p_0)(\partial_{x_k} p_1) + \nu |c^k|^{\gamma-2} c^k \phi^k \di x.
\end{align}
Using $p_0$ as a test function in \eqref{weakP} with the permeability tensor $\P[c+\epsilon\phi] = rI + c + \epsilon \phi$
%by $p_0$ and integration by parts, taking into account the no-flux boundary condition \eqref{noflux1},
gives
\begin{align*}
   \sum_{k=1}^d\int_{\Omega} \bl r+c^k+\epsilon\phi^k\br\left( \partial_{x_k} p_0\right)^2 + \epsilon (r+c^k) (\partial_{x_k} p_0) (\partial_{x_k} p_1) \di x
   = \int_{\Omega} Sp_0\di x + \mathcal{O}(\epsilon^2).
\end{align*}
Subtracting the identity
\(  \label{identity}
   \sum_{k=1}^d \int_{\Omega} (r+c^k) \left( \partial_{x_k} p_0\right)^2 \di x= \int_{\Omega} S p_0\di x,
\)
we obtain
\begin{align*}
   \sum_{k=1}^d\int_{\Omega} \left( \partial_{x_k} p_0\right)^2 \phi^k + (r+c^k) (\partial_{x_k} p_0)(\partial_{x_k} p_1) \di x = 0.
\end{align*}
Plugging this into \eqref{eq:firstvariation} gives
\begin{align*}
   \left.\tot{}{\epsilon} \E[c+\epsilon\phi]\right|_{\epsilon=0} =
   \sum_{k=1}^d\int_{\Omega} \left[ - \left( \partial_{x_k} p_0\right)^2 + \nu |c^k|^{\gamma-1} \right] \phi^k\di x,
\end{align*}
and \eqref{crit1} follows.
\end{proof}

From \eqref{crit1} it follows that for $k=1,\dots,d$ there exist disjoint Lebesgue-measurable sets $\mathcal{A}_k^+$, $\mathcal{A}_k^-$ such that
%$\overline\Omega = \overline{\mathcal{A}_k^+} \cup \overline{\mathcal{A}_k^-}$ for all $k=1,\dots,d$ and
$\Omega = {\mathcal{A}_k^+} \cup {\mathcal{A}_k^-}$ for all $k=1,\dots,d$ and
\(  \label{decomposition1}
   \partial_{x_k} p = (\chi_{\mathcal{A}_k^+}-\chi_{\mathcal{A}_k^-}) \sqrt{\nu} |c^k|^\frac{\gamma-1}{2}, \qquad \mbox{for } k=1,\dots,d.
\)
Note that by assumption $c_k \geq 0$.
Let us now choose an open subset $\Lambda \subset \Omega$ with boundary $\partial\Lambda\in C^{0,1}$
%and integrate \eqref{PcontR2} over $\Lambda$, using the Green's theorem,
and use the characteristic function of $\Lambda$ as the test function $\phi$ in \eqref{weakP}.
This formal calculation can be justified by using a smoothened version
of the characteristic function and passing to the limit,
adopting the assumption that the flux $q:=-(rI + c) \grad p$ is locally
Lipschitz continuous on $\Lambda$.
Then the generalized Green's formula by De Giorgi-Federer\cite{Evans-Gariepy}
gives
\(   \label{Plambda}
   - \int_{\partial\Lambda} n \cdot (rI + c) \grad p \d s = \int_\Lambda S \d x,
\)
where $n=n(x)$ denotes the outward unit normal vector to $\partial\Lambda$,
and $s=s(x)$ is the $(d-1)$-dimensional measure on $\partial\Lambda$.
Using \eqref{decomposition1}, we arrive at
\( \label{pre-Murray1}
   - \sqrt{\nu} \int_{\partial\Lambda} \sum_{k=1}^d (\chi_{\mathcal{A}_k^+}-\chi_{\mathcal{A}_k^-}) |c^k|^\frac{\gamma-1}{2} (r + c^k) n_k \d s
     = \int_\Lambda S \d x.
\)
We define the inflow and outflow segments of $\partial\Lambda$ as
\[
   \partial\Lambda^+:= \{ x\in\partial\Lambda;\, n(x)\cdot q(x) < 0 \},\qquad
   \partial\Lambda^- := \{ x\in\partial\Lambda;\, n(x)\cdot q(x) > 0 \},
\]
with the flux $q=-(rI + c) \grad p$.
Then, \eqref{pre-Murray1} can be written in the form
\( \label{Murray1}
  - \sqrt{\nu} \int_{\partial\Lambda^+} \left| \sum_{k=1}^d (\chi_{\mathcal{A}_k^+}-\chi_{\mathcal{A}_k^-}) |c^k|^\frac{\gamma-1}{2} (r + c^k) n_k \right| \d s && \\
  + \sqrt{\nu} \int_{\partial\Lambda^-} \left| \sum_{k=1}^d (\chi_{\mathcal{A}_k^+}-\chi_{\mathcal{A}_k^-}) |c^k|^\frac{\gamma-1}{2} (r + c^k) n_k \right| \d s 
     &=& \int_\Lambda S \d x.  \nonumber
\)
This is the Murray's law for the continuum model \eqref{PcontR2}--\eqref{Econt2}.% for $\gamma\neq 1$.
The two terms on the left-hand side represent the difference of in- and outgoing fluxes through the boundary $\partial\Lambda$,
which is balanced by the volume integral of the external in- and outflows $S=S(x)$.
Note the analogy with the Murray's law for the discrete system \eqref{Mur1}.
%Instead of formulating the Murray's law for a particular node of the discrete network,
%here it is formulated for a subset of the continuous domain $\Omega$.

Let us remark that the above formal calculation becomes
rigorous if we assume sufficient regularity of the flux $q=-(rI + c) \grad p$,
in particular, we need $q$ to be locally Lipschitz continuous on $\Lambda$.
In the next Section we construct solutions in the energy space which are less regular;
the proof of higher regularity is beyond the scope of this paper and we postpone it to a future work.

%%%%%%%%%%%%%%%%%%%%%%%%%%%%%%%%%%%%%%%%%%%%
\subsection{Construction of solutions}\label{subsec:constrution1}
From \eqref{crit1} it follows that %for $\gamma\neq 1$,
\[
   c^k = \nu^{-\frac{1}{\gamma-1}} (\partial_{x_k} p)^\frac{2}{\gamma-1}\qquad\mbox{for } k=1,\dots,d.
\]
Inserting this into \eqref{PcontR2}, we obtain
\(  \label{c-elliptic}
   - \sum_{k=1}^d \partial_{x_k} \left( \left( r + \nu^{-\frac{1}{\gamma-1}} (\partial_{x_k} p)^\frac{2}{\gamma-1} \right) \partial_{x_k} p \right) = S.
\)
The weak formulation of this nonlinear elliptic equation, equipped with a no-flux boundary condition,
reads
\(  \label{ell-weak}
   \int_\Omega \sum_{k=1}^d  \left( r + \nu^{-\frac{1}{\gamma-1}} (\partial_{x_k} p)^\frac{2}{\gamma-1} \right) (\partial_{x_k} p)  (\partial_{x_k} \phi) \d x = \int_\Omega S \phi \d x
\)
for all test functions $\phi\in C^\infty(\Omega)$. For $\gamma>1$ solutions can be constructed by the direct method of calculus of variations.

\begin{lemma}\label{lem:ell}
Let $\gamma>1$.
For every $S\in L^2(\Omega)$ there exists a unique solution $p\in H^1(\Omega)$ of \eqref{ell-weak}
satisfying $\int_\Omega p(x) \d x = 0$.
\end{lemma}

\begin{proof}
We define the functional $\mathcal{F}:H^1(\Omega)\to \mathbb{R}$ by
\begin{align}\label{f-ell-weak}
   \mathcal{F}[p]:=\int_\Omega \left(  \sum_{k=1}^d  \frac{r}{2}  |\partial_{x_k} p|^2 + \nu^{-\frac{1}{\gamma-1}} |\partial_{x_k} p|^\frac{2\gamma}{\gamma-1} \right) \d x
     - \int_\Omega p S \d x.
\end{align}
%and $\mathcal{F}[p]:=\infty$ if $\nabla p\notin L^{\frac{2\gamma-1}{\gamma-1}}$.
Since $\frac{2\gamma}{\gamma-1}>2$ for $\gamma>1$, the functional $\mathcal{F}$ is uniformly convex on $H^1(\Omega)$.
Moreover, a straightforward application of the Poincar\'{e} inequality provides coercivity on the set $H^{1,0}(\Omega):=\left\{ p\in H^1(\Omega);\; \int_\Omega p(x) \d x = 0 \right\}$.
The classical theory (see, e.g., Ref.~\cite{evans}) provides the existence of a unique minimizer $p\in H^{1,0}(\Omega)$.
Since, as is easy to check, \eqref{f-ell-weak} is the Euler-Lagrange equation for the minimizer of $\mathcal{F}$,
$p$ is the unique solution in $H^{1,0}(\Omega)$ of \eqref{ell-weak}.
\end{proof}

Let us note that \eqref{ell-weak} is independent of the choice of the sets $\mathcal{A}_k^+$, $\mathcal{A}_k^-$
in \eqref{decomposition1}, and so is the solution $p$ constructed in Lemma \ref{lem:ell}.
We can therefore write, for $k=1,\dots,d$,
\[
   \mathcal{A}_k^+ = \{x\in\Omega; \partial_{x_k} p(x) > 0\}, \qquad
   \mathcal{A}_k^- = \{x\in\Omega; \partial_{x_k} p(x) < 0\}.
\]

%%%%%%%%%%%%%%%%%%%%%%%%%%%%%%%%%%%%%%%%%%%%%%%%%%%%%%
\subsection{Murray law for the continuum model with diffusion}\label{subsec:diff}
Following Ref.~\cite{HKM}, we equip the model with a linear diffusive term that accounts
for random fluctuations in the medium.
In particular, we update the energy functional \eqref{Econt2} to
\(   \label{Econt3}
   \E[c] := \int_{\Omega} D^2 |\grad c|^2 + \grad p\cdot \P[c]\grad p + \frac{\nu}{\gamma} |c|^\gamma  \d x,
\)
where $D^2>0$ is the diffusivity constant and the symbol $|\grad c|^2$
is defined as $\sum_{k=1}^d |\grad c^k|^2$.
The permeability tensor is given by \eqref{P}, i.e., $\P[c] = rI + c$.
As can be easily checked, the Euler-Lagrange equations corresponding to critical
points of the functional \eqref{Econt3} constrained by the Poisson equation \eqref{weakP} read
\(   \label{crit1diff}
   \left( \partial_{x_k} p\right)^2 + D^2 \laplace c^k - \nu  |c^k|^{\gamma-1} = 0,\qquad\mbox{for } k=1,\dots,d.
\)
Note that for sufficiently regular critical points of \eqref{Econt3}, \eqref{weakP}
we have $\nu  |c^k|^{\gamma-1} - D^2 \laplace c^k \geq 0$ almost everywhere on $\Omega$.
Then it follows that for $k=1,\dots,d$ there exist disjoint Lebesgue-measurable sets $\mathcal{A}_k^+$, $\mathcal{A}_k^-$ such that
%$\overline\Omega = \overline{\mathcal{A}_k^+} \cup \overline{\mathcal{A}_k^-}$ for all $k=1,\dots,d$ and
$\Omega = {\mathcal{A}_k^+} \cup {\mathcal{A}_k^-}$ for all $k=1,\dots,d$ and
\(  \label{decomposition1diff}
   \partial_{x_k} p = (\chi_{\mathcal{A}_k^+}-\chi_{\mathcal{A}_k^-}) \left( \nu  |c^k|^{\gamma-1} - D^2 \laplace c^k \right)^{1/2}, \qquad \mbox{for } k=1,\dots,d.
\)
As before we choose an open subset $\Lambda \subset \Omega$ with boundary $\partial\Lambda\in C^{0,1}$,
and using \eqref{decomposition1diff} in \eqref{Plambda}, we obtain
\( \label{pre-Murray1diff}
   - \int_{\partial\Lambda} \sum_{k=1}^d (\chi_{\mathcal{A}_k^+}-\chi_{\mathcal{A}_k^-}) (r+c^k) \left( \nu  |c^k|^{\gamma-1} - D^2 \laplace c^k \right)^{1/2} n_k \d s
     = \int_\Lambda S \d x,
\)
where $n=n(x)$ denotes the outward unit normal vector to $\partial\Lambda$,
and $s=s(x)$ is the $(d-1)$-dimensional measure on $\partial\Lambda$.
We define the inflow and outflow segments of $\partial\Lambda$ as
\[
   \partial\Lambda^+:= \{ x\in\partial\Lambda;\, n(x)\cdot q(x) < 0 \},\qquad
   \partial\Lambda^- := \{ x\in\partial\Lambda;\, n(x)\cdot q(x) > 0 \},
\]
where we denoted the flux $q:=-(rI + c) \grad p$.
Then, \eqref{pre-Murray1diff} can be written in the form
\[ %\label{Murray1diff}
  - \int_{\partial\Lambda^+} \left| \sum_{k=1}^d (\chi_{\mathcal{A}_k^+}-\chi_{\mathcal{A}_k^-}) (r + c^k) \left( \nu  |c^k|^{\gamma-1} - D^2 \laplace c^k \right)^{1/2} n_k \right| \d s \\
  + \int_{\partial\Lambda^-} \left| \sum_{k=1}^d (\chi_{\mathcal{A}_k^+}-\chi_{\mathcal{A}_k^-}) (r + c^k) \left( \nu  |c^k|^{\gamma-1} - D^2 \laplace c^k \right)^{1/2} n_k \right| \d s \\
     = \int_\Lambda S \d x.  \nonumber
\]
This is the Murray's law for the continuum model \eqref{Econt3} constrained by the Poisson equation \eqref{weakP}.

%%%%%%%%%%%%%%%%%%%%%%%%%%%%%%%%%%%%%%%%%%%%%%%%%%%%%%
\def\bbP{\mathbb{P}}

\section{Murray's law for the phenomenological continuum model}\label{sec:cont2}
In this Section we consider the continuum model briefly proposed in Ref.~\cite{Hu}
and analyzed in Refs.~\cite{HMP15, HMPS16, AAFM}.
Its phenomenological derivation, carried out in Ref.~\cite{bookchapter},
assumes that the network domain $\Omega\subset\R^d$ is occupied by a porous medium
with the permeability tensor
\(   \label{bbP}
   \bbP[m] := m\otimes m
\)
with $m=m(x)\in\R^d$ a smooth vector field on $\R^d$.
Note that $\bbP[m]$ has the eigenvalue $|m|^2$ with eigenvector $m$,
and $0$ with eigenvectors orthogonal to $m$.
Thus, it represents conduction along the direction $m$ with conductivity $|m|^2$,
while there is no conduction in directions perpendicular to $m$.
Consequently, we locally identify the network conductivity with
the principal eigenvalue of $\bbP$, i.e.,  $C_i \simeq |m|^2$
in the neighborhood of a node $i\in\Vset$.
Assuming quasi-incompressibility of the fluid,
i.e., constant fluid density along particle trajectories, we have
the local mass conservation law
\(  \label{div v}
   \grad\cdot q = S \qquad\mbox{in } \Omega,
\)
where $q=q(x)\in\R^d$ is the flux
and $S=S(x)$ is the density of external sources and sinks;
see Ref.~\cite{bookchapter} for details.
Assuming the validity of Darcy's law for slow flow in porous media\cite{Whitaker},
we have
\[
   q = - \bbP[m]\grad p,
\]
where $p=p(x)$ is the fluid pressure.
Combining with \eqref{div v} and \eqref{bbP}, we arrive at the Poisson equation
\(  \label{Poisson2}
   - \grad\cdot( (m\otimes m)\grad p ) = S.
\)
Let us point out the possible strong degeneracy of \eqref{Poisson2}, leading to
solvability issues. That is why in Refs.~\cite{HMP15, HMPS16, AAFM} the following regularization
was considered,
\(  \label{r-Poisson}
   - \grad\cdot( (r I + m\otimes m)\grad p ) = S,
\)
where $I$ is the identity matrix and the scalar function $r=r(x) \geq r_0 > 0$ describes the isotropic background permeability of the medium.
Then, \eqref{r-Poisson} is uniformly elliptic and, equipped with suitable boundary conditions, solvable.
We shall assume no flux through the boundary $\partial\Omega$, i.e., 
equip \eqref{r-Poisson} with the homogeneous Neumann boundary condition
\(  \label{noflux2}
   n\cdot \P[m]\grad p = 0\qquad\mbox{on }\partial\Omega,
\)
where $n=n(x)$ denotes the outer unit normal vector to $\partial\Omega$.
Then the solution $p=p(x)$ is unique up to an additive constant and the flux $q = - \bbP[m]\grad p$ is defined uniquely.
%However, since in this paper we are only interested in formal derivation of the Murray's law,
%we stick to the original Poisson equation \eqref{Poisson2} and assume that a sufficiently regular solution $p=p[m]$
%exists, unique possibly up to an additive constant.

The local conductivity of the network being identified with $|m|^2$,
the continuum analogue of the discrete energy functional \eqref{eq:energydisc} is
\( \label{E2}
    \E[m] %&=& \frac12 \int (m\cdot\grad p[m])^2 + \frac{\nu}{\gamma} |m|^{2\gamma} \d x \\
       = \int \grad p[m]\cdot \bbP[m]\grad p[m] + \frac{\nu}{\gamma} |m|^{2\gamma} \d x,
\)
where $p[m]$ denotes a solution of the Poisson equation \eqref{r-Poisson} and $\P[m] := r I + m\otimes m$.
%The critical point of the energy \eqref{E2}, constrained by \eqref{r-Poisson}, is characterized as follows:

\begin{lemma}\label{lem:EL2}
%Let $\gamma>0$.
%The critical points $(m,p)$ of the constrained functional \eqref{E2}, \eqref{r-Poisson} satisfy
The Euler-Lagrange equations for the constrained energy minimization problem \eqref{r-Poisson}, \eqref{E2} read
\(   \label{crit2}
   (\grad p \otimes \grad p)m - \nu |m|^{2(\gamma-1)}m = 0.
\)
\end{lemma}

\begin{proof}
We calculate the first variation of $\E$ given by \eqref{E2} in the direction $\phi=\phi(x)\in\R^d$.
Using the expansion 
\begin{align}\label{eq:expansioncont}
   p[m+\epsilon\phi]=p_0+\epsilon p_1 + \mathcal{O}(\eps^2),
\end{align}
we have
\(\label{eq:firstvariation2}
   \left.\tot{}{\epsilon}\E[m+\epsilon\phi]\right|_{\epsilon=0} &=&
         \int_{\Omega} 2 r \grad p_0\cdot\grad p_1 + 2(\grad p_0\cdot m)(\grad p_1\cdot m) \\
         && \qquad  + 2(\grad p_0\cdot m)(\grad p_0\cdot\phi) + 2\nu |m|^{2(\gamma-1)}m\cdot\phi \d x.  \nonumber
\)
Multiplication of the Poisson equation \eqref{r-Poisson} by $p_0$ and integration by parts,
taking into account the no-flux boundary condition, gives
\begin{align*}
   \int_{\Omega} r |\grad p_0|^2 + |m\cdot\grad p_0|^2 + \eps r\grad p_0\cdot\grad p_1 + \eps (\grad p_0\cdot m)(\grad p_1\cdot m) 
      + 2\eps (\grad p_0\cdot m)(\grad p_0\cdot\phi) \di x \\
   = \int_{\Omega} Sp_0\di x + \mathcal{O}(\epsilon^2).
\end{align*}
Subtracting the identity
\[
   \int_{\Omega} r |\grad p_0|^2 + |m\cdot\grad p_0|^2 \di x= \int_{\Omega} S p_0\di x,
\]
we obtain
\begin{align*}
   \int_{\Omega} r\grad p_0\cdot\grad p_1 + (\grad p_0\cdot m)(\grad p_1\cdot m) \di x = -2 \int_\Omega (\grad p_0\cdot m)(\grad p_0\cdot\phi) \di x.
\end{align*}
Plugging this into \eqref{eq:firstvariation2} gives
\begin{align*}
   \left.\tot{}{\epsilon} \E[m+\epsilon\phi]\right|_{\epsilon=0} = 2\int_\Omega \bigl[- (\grad p_0\otimes \grad p_0)m + \nu |m|^{2(\gamma-1)}m \bigr] \cdot\phi \d x.
\end{align*}
\end{proof}

Interpreting \eqref{crit2} as an eigenvalue-eigenvector problem for the matrix $\grad p\otimes\grad p$,
and noting that the matrix has the eigenvalues $|\grad p|^2$ (with unit eigenvector $\frac{\grad p}{|\grad p|}$) and zero,
there exists a measurable set $\mathcal{A}\subset\Omega$ such that $m$ is a real multiple of $\grad p$ on $\mathcal{A}$
and zero otherwise.
Moreover, we have
\[
    \nu |m|^{2(\gamma-1)} = |\grad p|^2 \qquad\mbox{on } \mathcal{A}. 
\]
Therefore, there exist disjoint measurable sets $\mathcal{A}^+$, $\mathcal{A}^-$ such that
%$\overline\Omega = \overline{\mathcal{A}^+} \cup \overline{\mathcal{A}^-}$ and
${\mathcal{A}^+} \cup {\mathcal{A}^-} = \mathcal{A}$ and
\(  \label{decomposition2}
   \grad p = (\chi_{\mathcal{A}^+}-\chi_{\mathcal{A}^-}) \sqrt{\nu} |m|^{\gamma-2} m \qquad\mbox{on }\mathcal{A},
\)
where $\chi_{\mathcal{A}^+}$, resp., $\chi_{\mathcal{A}^-}$ are the characteristic functions of the sets
$\mathcal{A}^+$, resp., $\mathcal{A}^-$.
For $\gamma=1$ we define $\frac{m}{|m|}:= 0$ for $m=0$.
Note that the measurable sets $\mathcal{A}^+$, $\mathcal{A}^-$ can be chosen arbitrarily 
and each choice gives a solution of the Euler-Lagrange equations \eqref{crit2}.

%\begin{remark}
In Ref.~\cite{HMP15}, Theorem 7 and Remark 6, it was shown that if $\mathrm{meas}(\mathcal{A}) < \mathrm{meas}(\Omega)$,
i.e., if $m=0$ on a set of positive Lebesgue measure, then the solution $(m,p)$
is linearly unstable for the gradient flow system corresponding to \eqref{r-Poisson}--\eqref{E2}.
In other words, a necessary condition for asymptotic stability of the gradient flow system linearized at $(m,p)$
is that $\mathrm{meas}(\mathcal{A}) = \mathrm{meas}(\Omega)$.
We adopt this assumption for the sequel, since only the linearly stable steady states
are likely to be observed in the nature or as results of numerical simulations.
%\end{remark}

Let us now choose an open subset $\Lambda \subset \Omega$ with boundary $\partial\Lambda\in C^{0,1}$
and integrate \eqref{r-Poisson} over $\Lambda$, using the Green's theorem,
\( \label{pre-Murray2}
   - \int_{\partial\Lambda} n\cdot (r I + m\otimes m) \grad p \d s = \int_\Lambda S \d x,
\)
where $n=n(x)$ denotes the outward unit normal vector to $\partial\Lambda$,
and $s=s(x)$ is the $(d-1)$-dimensional measure on $\partial\Lambda$.
Here we assumed that the flux $q:=-(r I + m\otimes m) \grad p$
is locally Lipschitz continuous on $\Lambda$, so that the generalized Green's formula
by De Giorgi-Federer\cite{Evans-Gariepy} can be applied.
We define the inflow and outflow segments of $\partial\Lambda$ as
\[
   \partial\Lambda^+ := \{ x\in\partial\Lambda;\, n(x)\cdot q(x) < 0 \},\qquad
   \partial\Lambda^- := \{ x\in\partial\Lambda;\, n(x)\cdot q(x) > 0 \},
\]
with the flux $:=-(r I + m\otimes m) \grad p$.
Note that due to \eqref{decomposition2} we have
\[
   q = -(r I + m\otimes m) \grad p = -(r I + |m|^2) \grad p \qquad\mbox{on }\mathcal{A},
\]
so that $\sign(n\cdot q) = -\sign(n\cdot\grad p)$ and we can equivalently
(and slightly more elegantly)
define the inflow and outflow segments in terms of $\grad p$ only,
\[
   \partial\Lambda^+ := \{ x\in\partial\Lambda;\, n(x)\cdot \grad p(x) > 0 \},\qquad
   \partial\Lambda^- := \{ x\in\partial\Lambda;\, n(x)\cdot \grad p(x) < 0 \}.
\]
Inserting then \eqref{decomposition2} into \eqref{pre-Murray2} we obtain
\(  \label{Murray2}
   - \sqrt{\nu} \int_{\partial\Lambda^+} \left( r + |m|^2 \right) |m|^{\gamma-2} |m\cdot n| \d s &&\\
   + \sqrt{\nu} \int_{\partial\Lambda^-} \left( r + |m|^2 \right) |m|^{\gamma-2} |m\cdot n| \d s
   &=& \int_\Lambda S \d x.
   \nonumber
\)
This is the Murray's law for the linearly stable solutions of the phenomenological continuum model \eqref{r-Poisson}--\eqref{E2}.
Clearly, the two terms on the left-hand side represent the difference of in- and outgoing fluxes
through the boundary $\partial\Lambda$, which is balanced by the volume integral of the external in- and outflows $S=S(x)$
on the right-hand side.
Note the analogy with the Murray's law for the discrete system \eqref{Mur1}.
%Instead of formulating the Murray's law for a particular node of the discrete network,
%here it is formulated for a subset of the continuous domain $\Omega$.

\begin{remark}
Should the restriction to only consider linearly stable states be relaxed,
i.e., should it be admitted that $\mathrm{meas}(\mathcal{A}) < \mathrm{meas}(\Omega)$,
then \eqref{Murray2} must be modified.
Namely, since we have no information about $\grad p$ on the complement $\mathcal{A}^c$ of $\mathcal{A}$ in $\Omega$,
\eqref{Murray2} needs to be rewritten as
\[
   - \sqrt{\nu} \int_{\mathcal{A}\cap\partial\Lambda^+} \left( r + |m|^2 \right) |m|^{\gamma-2} |m\cdot n| \d s
   - \int_{\mathcal{A}^c\cap\partial\Lambda^+} r |\grad p\cdot n| \d s  &&\\
   + \sqrt{\nu} \int_{\mathcal{A}\cap\partial\Lambda^-} \left( r + |m|^2 \right) |m|^{\gamma-2} |m\cdot n| \d s 
   + \int_{\mathcal{A}^c\cap\partial\Lambda^+} r |\grad p\cdot n| \d s
   &=& \int_\Lambda S \d x,
\]
where we used the fact that $m\equiv 0$ on $\mathcal{A}^c$.
\end{remark}

Let us again remark that the central assumption for the above calculations
to be rigorously valid is that the flux $q=-(r I + m\otimes m) \grad p$
is locally Lipschitz continuous on $\Lambda$.
%In the following Section we construct solutions with lower regularity
%and we postpone the higher regularity analysis to a future work.

%%%%%%%%%%%%%%%%%%%%%%%%
\subsection{Construction of solutions}
We provide a short overview of rigorous existence results for the system \eqref{r-Poisson}, \eqref{crit2}
that were obtained in Refs.~\cite{HMP15} and \cite{HMPS16}.
Interpreting equation \eqref{crit2} as an eigenvalue problem for the matrix $\grad p \otimes\grad p$,
we have the following expression for $m=m(x)$,
\(  \label{m0}
    m(x) := \left(\chi_{\mathcal{A}_+}(x) - \chi_{\mathcal{A}_-}(x)\right) \nu^\frac{1}{2(1-\gamma)} |\grad p(x)|^\frac{2-\gamma}{\gamma-1} \grad p(x),
\)
where $\mathcal{A}_+\subseteq\Omega$ and $\mathcal{A}_-\subseteq\Omega$ are measurable disjoint sets.
Inserting \eqref{m0} into \eqref{r-Poisson} gives
\(   \label{p0}
   -\grad\cdot\left[\left(r I + \nu^\frac{1}{1-\gamma} |\grad p(x)|^\frac{2}{\gamma-1}\chi_{\mathcal{A}_+\cup\mathcal{A}_-}(x) \right)\grad p(x) \right] = S,
\)
subject to the no-flux boundary condition.
Solutions of \eqref{p0} for $\gamma>1$ are constructed by the direct method of calculus of variations.

\begin{lemma}\label{lem:calcVar2}
Let $\gamma>1$.
Then for any $S\in L^2(\Omega)$ and for any pair of measurable disjoint sets $\mathcal{A}_+$, $\mathcal{A}_-\subseteq\Omega$
there exists a unique (up to an additive constant) weak solution $p\in H^1(\Omega) \cap W^{1,2\gamma/(\gamma-1)}(\mathcal{A}_+\cup\mathcal{A}_-)$ of \eqref{p0}.
\end{lemma}

\begin{proof}
Direct method of calculus of variations, see Theorem 6 of Ref.~\cite{HMP15}.
\end{proof}

In the case $\gamma=1$, \eqref{crit2} reads
\[   %\label{stat_gamma=1}
   (\grad p\otimes\grad p) m  = \nu m,
\]
i.e., $m$ is either the zero vector or an eigenvector of the matrix $\grad p\otimes\grad p$
with eigenvalue $\nu>0$.
The spectrum of $\grad p\otimes\grad p$ consists of zero and $|\grad p|^2$,
so that $m\neq 0$ is only possible if $|\grad p|^2 = \nu$.
Therefore, for every stationary solution there exists a measurable function $\lambda=\lambda(x)$
such that
\[
   m(x) = \lambda(x)\chi_{\{|\grad p|^2=\nu\}}(x) \grad p(x)
\]
and $p$ solves the highly nonlinear Poisson equation
\(   \label{hnp}
   -\grad\cdot\left[ \left(1+ \lambda(x)^2 \chi_{\{c^2|\grad p|^2=\nu\}}(x) \right) \grad p \right] = S.
\)
%subject to the homogeneous Dirichlet boundary condition $p=0$ on $\partial\Omega$.
In Ref.~\cite{HMPS16}, Section 4.2, it was shown that \eqref{hnp} is equivalent to the free boundary problem
\(
   -\grad\cdot\left[(1+a(x)^2)\grad p\right] &=& S,\qquad p \in H^1_0(\Omega), \label{NLP1} \\
   |\grad p(x)|^2 &\leq& \nu,\qquad \mbox{a.e. on }\Omega, \label{NLP2} \\
   a(x)^2 \left[|\grad p(x)|^2-\nu \right] &=& 0,\qquad \mbox{a.e. on }\Omega, \label{NLP3}
\)
for some measurable function $a^2=a(x)^2$ on $\Omega$ which is the Lagrange multiplier
for the condition \eqref{NLP2}.
%The function $\lambda=\lambda(x)$ can be chosen as $\lambda(x):=ca(x)$.
Solutions of \eqref{NLP1}--\eqref{NLP3} can be obtained as minimizers of the energy functional
\[
    \mathcal{F}[p]:=\int_\Omega \left( \frac{|\grad p|^2}{2} - Sp  \right) \d x
\]
on the set $\{p\in H_0^1(\Omega), c^2|\grad p|^2\leq 1 \mbox{ a.e. on } \Omega\}$,
or using the penalty method; see Section 4.2 of Ref.~\cite{HMPS16} for details.

Finally, for $1/2 \leq \gamma < 1$ the above approach fails due to the singularity of the term $|\grad p(x)|^\frac{2-\gamma}{\gamma-1} \grad p(x)$
at $|\grad p|=0$ and the resulting non-boundedness from below of the associated functional.
However, stationary solutions can be constructed by ``cutting off'' small values of $|\grad p|$, see Section 4.3 of Ref.~\cite{HMPS16} for details.

%%%%%%%%%%%%%%%%%%%%%%%%%%%%%%%%%%%%%%%%%%%%%%%%%%%%%
\subsection{Murray's law for the phenomenological continuum model with diffusion}\label{subsec:diff2}
Following Refs.~\cite{HMP15} and \cite{HMPS16}, we equip the model with a linear diffusive term that accounts
for random fluctuations in the medium.
In particular, we update the energy functional \eqref{E2} to
\( \label{E3}
    \E[m] = \int D^2 |\grad m|^2 + \grad p[m]\cdot \bbP[m]\grad p[m] + \frac{\nu}{\gamma} |m|^{2\gamma} \d x,
\)
where $D^2>0$ is the diffusivity constant, $p[m]$ denotes a solution of the Poisson equation \eqref{r-Poisson} and $\P[m] = r I + m\otimes m$.
As can be easily checked, the Euler-Lagrange equations corresponding to critical
points of the functional \eqref{E3} constrained by the Poisson equation \eqref{r-Poisson} read
\[
    (\grad p\otimes\grad p) m = - D^2 \laplace m + \nu |m|^{2(\gamma-1)} m.
\]
Note that for sufficiently regular critical points of \eqref{E3}, \eqref{r-Poisson}
we have $(- D^2 \laplace m + \nu |m|^{2(\gamma-1)} m)\cdot m \geq 0$ almost everywhere in $\Omega$.
Let us denote
\[
   \mathcal{A} := \left\{ x\in\Omega;\, \bigl(- D^2 \laplace m(x) + \nu |m|^{2(\gamma-1)} m(x)\bigr)\cdot m(x) > 0 \right\}.
\]
Then, there exist disjoint measurable sets $\mathcal{A}^+$, $\mathcal{A}^-$ such that
${\mathcal{A}^+} \cup {\mathcal{A}^-} = \mathcal{A}$ and
\(  \label{decomposition2diff}
   \grad p = (\chi_{\mathcal{A}^+}-\chi_{\mathcal{A}^-}) \left( \nu |m|^{2\gamma} - D^2m\cdot \laplace m \right)^{-1/2}
      \left( \nu |m|^{2(\gamma-1)} m - D^2 \laplace m \right),
\)
where $\chi_{\mathcal{A}^+}$, resp., $\chi_{\mathcal{A}^-}$ are the characteristic functions of the sets
$\mathcal{A}^+$, resp., $\mathcal{A}^-$.
On the set $\mathcal{A}$ we have $(\grad p\otimes\grad p) m = 0$, which is equivalent to $m\cdot\grad p = 0$.
%\blueJH{Shall we make a comment about the case $m=\laplace m=0$? This may happen, should(?) only happen pointwise..}

Let us now choose an open subset $\Lambda \subset \Omega$ with boundary $\partial\Lambda\in C^{0,1}$
and define the inflow and outflow segments of $\partial\Lambda$ as
\[
   \partial\Lambda^+ := \{ x\in\partial\Lambda;\, n(x)\cdot q(x) < 0 \},\qquad
   \partial\Lambda^- := \{ x\in\partial\Lambda;\, n(x)\cdot q(x) > 0 \},
\]
where we denoted the flux $q:=-(r I + m\otimes m) \grad p$.
Using \eqref{decomposition2diff} in \eqref{pre-Murray2} leads to
\[  %\label{Murray2diff}
   - \int_{\mathcal{A}\cap\partial\Lambda^+}  \left( \nu |m|^{2\gamma} - D^2m\cdot \laplace m \right)^{-1/2} \left|
        n\cdot \left( rI + m\otimes m \right) \left( \nu |m|^{2(\gamma-1)} m - D^2 \laplace m \right)\right|  \d s \\
   - \int_{\mathcal{A}^c\cap\partial\Lambda^+}  r |\grad p\cdot n| \d s \\ 
   + \int_{\mathcal{A}\cap\partial\Lambda^-}  \left( \nu |m|^{2\gamma} - D^2m\cdot \laplace m \right)^{-1/2}  \left|
       n\cdot \left( rI + m\otimes m \right) \left( \nu |m|^{2(\gamma-1)} m - D^2 \laplace m \right) \right| \d s \\
   + \int_{\mathcal{A}^c\cap\partial\Lambda^-}  r |\grad p\cdot n| \d s \\ 
   = \int_\Lambda S \d x.
\]
This is the Murray's law for the phenomenological continuum model \eqref{E2} constrained by the Poisson equation \eqref{r-Poisson}.

\section*{Acknowledgment}
Giulia Pilli acknowleges support from the Austrian Science Fund (FWF) through the grants F 65 and W 1245.

\end{document}